\documentclass{amsart}

\usepackage{mathptmx, amscd, amssymb, amsmath, enumerate, slashbox, hyperref, float}

\usepackage[utf8]{inputenc}
\usepackage[all]{xy}
\usepackage{color}
\usepackage{hyperref}

\usepackage{amsmath, mathtools}
\usepackage{hyperref}
\hypersetup{colorlinks=true,linkcolor={blue}}
\usepackage{cleveref}
\usepackage{extpfeil}
\usepackage{amssymb}
\usepackage{amsfonts}
\usepackage{graphicx}
\usepackage{mathrsfs}
\usepackage{stmaryrd}
\usepackage{stmaryrd}
\usepackage{mathtools}
\usepackage{amscd}
\usepackage[all]{xy}
\usepackage{hyperref}
\usepackage{amsthm, enumitem }
\usepackage{cleveref}
\usepackage{verbatim}
\usepackage{amscd}
\usepackage{mathtools}
\usepackage{comment}
\usepackage{diagbox,bm}
\usepackage{tikz-cd}

\usepackage{tikz}
\usetikzlibrary{matrix,arrows,decorations.pathmorphing,cd}

\newtheorem{theoremx}{\bf Theorem}

\newtheorem{thm}{\bf Theorem}[section]

\newtheorem{lemma}[thm]{\bf Lemma}
\newtheorem{cor}[thm]{\bf Corollary}

\theoremstyle{definition}

\theoremstyle{remark}
\newtheorem{fact}[thm]{\bf Fact}
\newtheorem{remark}[thm]{\bf Remark}

\newtheorem{question}[thm]{\bf Question}

\numberwithin{equation}{section}

\newcommand{\id}{\operatorname{id}}
\newcommand{\Ext}{\operatorname{Ext}}

\newcommand{\Tor}{\operatorname{Tor}}
\newcommand{\Hom}{\operatorname{Hom}}

\newcommand{\chara}{\operatorname{char}}
\newcommand{\codim}{\operatorname{codim}}

\newcommand{\m}{\mathfrak{m}}

\newcommand{\fm}{\mathfrak{m}}
\newcommand{\ls}{\leqslant}
\newcommand{\gs}{\geqslant}

\DeclareMathOperator{\Soc}{Soc}
\DeclareMathOperator{\Coker}{Coker}
\DeclareMathOperator{\Ker}{Ker}
\DeclareMathOperator{\Image}{Im}
\DeclareMathOperator{\soc}{soc}

\DeclareMathOperator{\cx}{cx}

\newcommand{\w}{\omega}

\def \ZZ{\mathbb Z}

\dedicatory{}

\begin{document}
\title[Exterior powers and Tor-persistence]{Exterior powers and Tor-persistence}

\author[Justin Lyle]{Justin Lyle}
\address{Justin Lyle\\ Department of Mathematical Sciences \\ University of Arkansas \\ 525 Old Main \\ University of Arkansas\\ Fayetteville, Arkansas 72701}
\email{jl106@uark.edu}
 
\author[Jonathan Monta\~no]{Jonathan Monta\~no$^*$}
\address{Jonathan Monta\~no \\ Department of Mathematical Sciences  \\ New Mexico State University  \\PO Box 30001\\Las Cruces, NM 88003-8001}
\email{jmon@nmsu.edu}
\urladdr{https://web.nmsu.edu/~jmon/}
\thanks{$^{*}$ The second author is  supported by NSF Grant DMS \#2001645.}

\author[Keri Sather-Wagstaff]{Keri Sather-Wagstaff}
\address{Keri Sather-Wagstaff \\ School of Mathematical and Statistical Sciences \\ Clemson University \\ O-110 Martin Hall \\ Box 340975 \\ Clemson, S.C. 29634 USA}
\email{ssather@clemson.edu}
\urladdr{https://ssather.people.clemson.edu/}

\begin{abstract}
A  commutative Noetherian ring $R$ is said to be Tor-persistent if, for any finitely generated $R$-module $M$, the vanishing of $\Tor_i^R(M,M)$ for $i\gg 0$ implies $M$ has finite projective dimension. An open question of Avramov, et. al. asks whether  any such $R$ is Tor-persistent. In this work, we exploit properties of exterior powers of modules and complexes to provide several  partial answers to this question; in particular, we show that every local ring $(R,\m)$ with $\m^3=0$ is Tor-persistent. As a consequence of our methods, we provide a new proof of the Tachikawa  Conjecture for positively graded  rings over a field  of characteristic different from 2. 
  \end{abstract}

\keywords{
Exterior squares,
Tachikawa's Conjecture,
Tor-persistence}

\subjclass[2010]{
13D07, 
13C10, 
13D02. 
}

\maketitle

\section{Introduction}


Several conjectures and open questions on the rigidity of $\Ext$ and $\Tor$ have recently gained  much attention. Among the most well-known of these is the Auslander-Reiten conjecture which poses that given  a commutative Noetherian ring $R$ and a finitely generated $R$-module $M$,  the vanishing of $\Ext^i_R(M,M \oplus R)$ for all $i>0$ forces $M$ to be projective \cite{AR75}. The Auslander-Reiten conjecture traces its roots to the representation theory of Artin algebras where it is intimately connected to Nakayama's conjecture and its generalized version. A significant special case of the Auslander-Reiten conjecture is the Tachikawa conjecture which posits that the Auslander-Reiten conjecture holds when $R$ is Cohen-Macaulay and $M=\w_R$ is a canonical module of $R$ \cite{avramov:edcrcvct}. Inspired by work of \c{S}ega \cite{sega:stfcar}, Avramov et. al. introduce the following which can be thought of as a 
version of the Auslander-Reiten conjecture for  Tor.

\begin{question}[{\cite{avramov:phcnr}}]\label{torpq}
Let $R$ be a commutative Noetherian ring. If, for a finitely generated $R$-module $M$, we have $\Tor^R_i(M,M)=0$ for $i \gg 0$, must $M$ have finite projective dimension?
\end{question}

Rings for which Question \ref{torpq} has an affirmative answer are called {\it $\Tor$-persistent}. Thus Question \ref{torpq} can be rephrased to ask whether every commutative Noetherian ring is $\Tor$-persistent. Several classes of rings are known to be $\Tor$-persistent, for example, complete intersection rings, Golod rings, and rings of small embedding codimension or multiplicity \cite{AB00,avramov:phcnr,Jo992,LM20}. The complete intersection case depends on support theory, which is only available in this setting, while the other known results depend on conditions for the vanishing of $\Tor_i^R(M,N)$ for all $i\gg 0$ and every $M$ and $N$, an approach that does not extend to the general case (see \cite{avramov:phcnr}).

The main purpose of this work is to provide evidence, and new insights, that this question may have an affirmative answer.  Our first main result provides a new class of rings which are Tor-persistent (see Theorem \ref{cubethm}).

\begin{theoremx}\label{cubethmintro}
If $(R,\mathfrak{m})$ is a local ring with $\m^3=0$, then $R$ is Tor-persistent.
\end{theoremx}


Our second main result provides some restrictions in the graded case; here $e_R(M)$ denotes the Hilbert-Samuel mutiplicity of the $R$-module $M$ (see Theorem \ref{bigthm}).

\begin{theoremx}\label{bigthmintro}

 Let $R=\oplus_{i\gs 0} R_i$ be a Noetherian standard graded  algebra over a field $R_0=k$ with  $\chara k \ne 2$. Let $M$ be a finitely generated graded $R$-module satisfying the following.
\begin{enumerate}
\item[$(1)$] $e_R(M)=e_R(R)$,

\item[$(2)$]  $\Tor_i^R(M,M)=0$ for $i>0$, and

\item[$(3)$]  $M \otimes_R M$ has no embedded primes.
\end{enumerate}
Then $M \cong R$.
\end{theoremx}

In fact, we prove Theorem \ref{bigthmintro} under more general assumptions. As a consequence of this result, we provide a commutative algebra  proof of the Tachi\-kawa  Conjecture for positively graded  rings over a field  of characteristic different from 2  (see Corollary \ref{cor200627a}). This result also follows from work of Zeng using techniques in  representation theory of Artin algebras  \cite{Zeng90} .

  Our approach to both of our main results provides an explanation as to why the vanishing of $\Tor_i^R(M,N)$ is special when $M=N$; namely, the vanishing of $\Tor^R_i(M,M)$ has consequences for the exterior and symmetric powers of $M$. For the cases we consider, these consequences come in the form of numerical constraints on the exterior and symmetric  squares, and are enough to conclude that the module in question is~free.

We conclude this section with some notation that we use in the subsequent ones. 
Through the remainder of the paper, let $(R,\m,k)$ be a commutative Noetherian ring which is either local or positively graded over the field $k$ with maximal homogeneous ideal $\m$.  If $R$ has a canonical module, it is denoted by $\omega_R$. We let $M$ be a finitely generated $R$-module; in the graded case we assume $M$ is homogeneous. We use $\nu_R(M)=\beta^R_0(M)$ for the minimal number of generators of $M$ and $l_R(M)$ for the length of $M$.   We write $\Omega^R_i(M)$ for the $i$th syzygy of  $M$ and $\beta^R_i(M)$ for the $i$th Betti number. We let $\codim R:=\nu_R(\m)-\dim R$ be the embedding codimension of $R$. 
We let $\iota_M:\bigwedge^2_R(M) \to M \otimes_R M$ be the antisymmetrization map defined on elementary wedges by $x_1 \bigwedge x_2 \mapsto x_1 \otimes x_2-x_2 \otimes x_1$.

\section{Tor-Persistence for Rings with Radical Cube Zero}\label{sec200627b}



The following main result contains Theorem \ref{cubethmintro} from the introduction.

\begin{thm}\label{cubethm}
Assume $(R,\m,k)$ is a local ring with $\m^3=0$.  If $M$ is an $R$-module such that $\Tor^R_i(M,M)=0$ for 
$2 \ls i \ls 5$, then $M$ is free.
\end{thm}

\begin{proof}
We may assume $M$ is nonzero. As a notational convenience, we set $\gamma_R(M)=\frac{l_R(M)}{\nu_R(M)}-1$. We note that $\gamma_R(M) \gs 0$ with equality if and only if $M \cong k^n$ for some $n$. 

Suppose $M$ is not free. Set $N=\Omega^R_1(M)$, $L=\Omega_2^R(M)$, and  $b=\nu_R(N)$. Let $\varphi$ be the map fitting in the natural exact sequence $0 \rightarrow L \xrightarrow{\varphi} R^b \rightarrow N \rightarrow 0$. Since $N \hookrightarrow \m R^{\nu_R(M)}$ and since $\m^3=0$, we have $\m^2N=0$. Similarly, we have $\m^2L=0$.  By dimension shifting, $\Tor^R_1(N,L)=0$, and so we have $\m(L \otimes_R L)=0$ by 
 \cite[Lemma 1.4]{HS04}.  Further, we have $\Tor^R_i(N,N)=0$ for $i=1,2,3$ so \cite[Theorem 2.5]{HS04} gives 
\begin{enumerate}
\item \label{item200627a} $\nu_R(L)=\gamma_R(N)b$,
\item \label{item200627b} $\nu_R(\m)=2\gamma_R(N)$, and
\item \label{item200627c} $r(R)=\gamma_R(N)^2$, where $r(R):=\dim_k \Soc(R)$ is the  type of $R$. 
\end{enumerate}
The map $\iota_{L \otimes_R k}\colon \bigwedge^2_k(L \otimes_R k)\to(L \otimes_R k)\otimes_k(L \otimes_R k)$ is injective because $L \otimes_R k$ is a $k$-vector space, and this map 
is naturally identified with $\iota_L \otimes \id_k\colon \bigwedge^2_R(L) \otimes_R k\to(L\otimes_RL)\otimes_Rk$.
As $L \otimes_R L$ is a $k$-vector space, so is its quotient $\bigwedge^2_R(L)$, 
hence $\iota_L \otimes \id_k$ is naturally identified with $\iota_L$.  In particular, $\iota_L$ is injective.

Now, we have $\varphi \otimes \varphi=(\varphi \otimes \id_{R^b}) \circ (\id_L \otimes \varphi)$.  The map $\varphi \otimes \id_{R^b}$ is injective since $\Tor_1^R(N,R^b)=0$, and $\id_L \otimes_R \varphi$ is injective since $\Tor^R_1(L,N)=\Tor^R_4(M,M)=0$. Thus $\varphi \otimes \varphi$ is also injective.  

Next, we have the following commutative diagram
\[\begin{tikzcd}
\bigwedge^2_R(L) \arrow[d, "\bigwedge^2_R(\varphi)"'] \arrow[r, "\iota_L"] & L \otimes_R L \arrow[d, "\varphi \otimes \varphi"] \\
\bigwedge^2_R(R^b) \arrow[r, "\iota_{R^b}"]                             & R^b \otimes_R R^b                                 
\end{tikzcd}\]
Since $\varphi \otimes \varphi$ and $\iota_L$ are both injective,
the commutivity of the diagram forces $\bigwedge^2_R(\varphi)$ to be injective.  Since $\bigwedge^2_R(L)$ is a $k$-vector space, it must thus embed in the socle of $\bigwedge^2_R(R^b)$.

The vector space dimension of $\bigwedge^2_R(L)$ is 
$$\displaystyle {\nu_R(L) \choose 2}={\gamma_R(N)b \choose 2}=\dfrac{\gamma_R(N)b(\gamma_R(N)b-1)}{2}$$ while that of $\soc(\bigwedge_R^2(R^b))$ is 
$$\displaystyle r(R){b \choose 2}=\gamma_R(N)^2\left(\dfrac{b(b-1)}{2}\right).$$ It follows that we must have 

\[\dfrac{\gamma_R(N)b(\gamma_R(N)b-1)}{2} \ls \gamma_R(N)^2\left(\dfrac{b(b-1)}{2}\right).\]  
If $\gamma_R(N)=0$, then $N$ is a $k$-vector space which cannot be, since $N$ has infinite projective dimension and $\Tor^R_1(N,N)=0$.
Therefore, as $b\neq 0$, we have 
\[\gamma_R(N)b-1 \ls \gamma_R(N)b-\gamma_R(N)\]
which forces $\gamma_R(N)=1$.  Thus $R$ is Gorenstein with $\nu_R(\m)=2$, by items~\eqref{item200627a}--\eqref{item200627b} above, and so $R$ is also a complete intersection. 

Let $\cx_R(M)$ be the complexity of $M$, since $R$ is a complete intersection, we  have $\cx_R(M) \ls \codim R=2$; see e.g. 
 \cite[Theorem 1.1 and subsequent paragraph]{Jo99}.
Thus, \cite[Proposition 2.3]{Jo99} forces $\Tor^R_i(M,M)=0$ for $i>0$.  By \cite[Theorem 4.2]{AB00}, this contradicts the fact that $M$ is not free. The  result follows.
\end{proof}

\section{Some Restrictions for Tor-persistence}\label{sec200627a}

Unless otherwise stated, throughout this section we let $R=\oplus_{i\gs 0} R_i$ be a positively graded algebra over a field $R_0=k$ and $\fm=\oplus_{i> 0} R_i$   its homogeneous maximal ideal. 
Let  $M=\oplus_{i\in \ZZ} M_i$
 be a finitely generated graded  $R$-module. 
The {\it Hilbert series} of $M$ is
  $$H_M(t)=\sum_{i\in \ZZ} (\dim_k M_i) t^i. $$   
  We recall that if $M\neq 0$, and if $\underline{x}=(x_1,\dots,x_{\dim M})$ is a homogeneous system of parameters on $M$ with $\deg x_i=a_i$, then there exists a Laurent polynomial $\varepsilon^{\underline{x}}_M(t)\in \ZZ[t, t^{-1}]$  with $\varepsilon^{\underline{x}}_M(1)>0$ such that $H_M(t)$ can be written as 
  $$H_M(t)=\frac{\varepsilon^{\underline{x}}_M(t)}{\prod_{i=1}^{\dim M}(1-t^{a_i})}$$ \cite[Proposition  4.4.1 and Remark 4.4.2]{BH93}. 

 We denote by $e_R(M)$ the {\it Hilbert-Samuel multiplicity} of $M$ $$e_R(M)=\lim_{n\to \infty }\frac{(\dim R)! l_R(M/\fm^nM)}{n^{\dim R}};$$ we also write $e(R)$ for $e_R(R)$. 
  
  For a graded complex  of finite rank graded free $R$-modules
$$X=\cdots X_{i+1}\to X_i\to X_{i-1}\to \cdots  $$ 
with $X_i=\oplus_{j\in \ZZ}R(-j)^{b_{i,j}}$, we denote by $$P_X(t,z)=\sum_{i,j\in 
\ZZ}  b_{i,j}t^jz^i $$ the {\it (graded) Poincar\'e series} of $X$. If  
 $F$ is a graded free resolution of $M$, then  we 
set  $P_M(t,z):=P_F(t,z)$.  The  additivity of length gives the following
comparison of Hilbert and Poincar\'e series.

\begin{fact}\label{HilbertPoincare}
For $R$ and $M$ as above, we have 
$H_M(t)=H_R(t)P_M(t,-1).$
\end{fact}

We now describe a construction of Buchsbaum-Eisenbud~\cite{buchsbaum:asffr}, 
following the presentation of Frankild-Sather-Wagstaff-Taylor\cite{FSWT08}.
Assume for the remainder of this paragraph
that $\chara k \ne 2$. Let $X$ be as above, and let $\alpha^X:X\otimes_R X\to X\otimes X$ be the map  defined on homogeneous generators by 
$$\alpha^X(x\otimes x')=x\otimes x'-(-1)^{|x||x'|}x'\otimes x.$$
Let $S^2_R(X)$ be the complex $\Coker(\alpha^X)$ and call it the {\it second symmetric power of X}.

 In the following statement we summarize some important properties of $S^2_R(X)$. We remark that although the statements in \cite{FSWT08} are in the local case, the arguments therein readily extend to account for the grading in $R$.
 
\begin{fact}[{\cite[3.8, 4.1, 3.12]{FSWT08}}]\label{turmeric}
Assume  $\chara k \ne 2$. Let $X$ be a graded complex of finite rank graded free $R$-modules.  
\begin{enumerate}
\item \label{turmeric1} The following exact sequences are split exact.
\begin{gather*}
0\to \Ker(\alpha^X)\to 
X\otimes_R X\to \Image(\alpha^X)\to 0\\
0\to \Image(\alpha^X)\to X\otimes_R X \to S^2_R(X)\to 0
\end{gather*}
\item \label{turmeric2}  $H_{0}( S^2_R(X))\cong S^2_R(H_0(X))$.
\item \label{turmeric3}  $P_{S^2_R(X)}(t,z)=\frac{1}{2}\Big[P_{X}(t,z)^2+P_{X}(t^2,-z^2)\Big]$
\end{enumerate}
\end{fact}

Now, assume $\chara k \ne 2$ and consider the antisymmetrization map  $ \iota_M\colon \bigwedge^2_R(M)\to M\otimes_R M$
defined in the introduction.
This map is  a split injection where the splitting map is given by $x\otimes y \mapsto \frac{1}{2} x\bigwedge y$. Since $\Coker(\iota_M)=S^2_R(M)$,  we have the following fact. 

\begin{fact}\label{split}
Assume  $\chara k \ne 2$, $M\otimes_R M \cong S^2_R(M)\oplus  \bigwedge^2_R (M)$
\end{fact}

The following lemma is essential in the proof of our main result.

\begin{lemma}\label{hilbertformula}
Assume  $\chara k \ne 2$. If $M$ is a graded $R$-module such that $\Tor_i^R(M,M)=0$ for every $i>0$, then we have
 $$H_{S^2_R(M)}(t) = \frac{H_{M}^2(t)}{2H_{R}(t)}+ \frac{H_{M}(t^2)H_{R}(t)}{2H_{R}(t^2)}\qquad\mbox{and} \qquad H_{\bigwedge^2_R (M)}(t) = \frac{H_{M}^2(t)}{2H_{R}(t)}- \frac{H_{R}(t)H_{M}(t^2)}{2H_{R}(t^2)}.$$
\end{lemma}
\begin{proof}
Let $F$ be a minimal graded free resolution of $M$. By the vanishing of Tor assumption,
the complex $F\otimes_R F$ is acyclic and therefore a minimal free resolution of $M\otimes_R M$. Therefore, Fact~\ref{turmeric}\eqref{turmeric1}--\eqref{turmeric2} imply 
that $S^2_R(F)$ is acyclic and  a minimal free resolution of $S^2_R(M)$. 
From Facts~\ref{HilbertPoincare}  and~\ref{turmeric}\eqref{turmeric3} we obtain
\begin{align*}
H_{S^2_R(M)}(t)=H_R(t)P_{S^2_R(F)}(t,-1)
&=\frac{H_R(t)}{2}\Big[P_{F}(t,-1)^2+P_{F}(t^2,-1)\Big]\\
&=\frac{H_{M}^2(t)}{2H_{R}(t)}+\frac{H_{R}(t)H_R(t^2)P_{F}(t^2,-1)}{2H_{R}(t^2)}\\
&=  \frac{H_{M}^2(t)}{2H_{R}(t)}+ \frac{H_{R}(t)H_{M}(t^2)}{2H_{R}(t^2)}.
\end{align*}
We note that 
$$H_{M\otimes_R M}(t)=H_R(t)P_{F\otimes_R F}(t,-1)=H_R(t)P_{F}^2(t,-1)=\frac{H_{M}^2(t)}{H_{R}(t)}.$$
Thus, it suffices to show $H_{M\otimes_R M}(t)=H_{S^2_R(M)}(t)+ H_{\bigwedge^2_R (M)}(t)$, which follows from Fact \ref{split}.
\end{proof}

We are now ready to prove Theorem \ref{bigthmintro}.

\begin{thm}\label{bigthm}

Let $R$ be a Noetherian positively graded $k$-algebra with $\chara k \ne 2$. Let $M$ be  a finitely generated graded $R$-module such that $\dim(M)=\dim(R)$ and  satisfying the following:

\begin{enumerate}

\item[$(1)$]  For some homogeneous system of parameters $\underline{x}$ of $R$, $\varepsilon^{\underline{x}}_M(1)=\varepsilon^{\underline{x}}_R(1)$.

\item[$(2)$]  $\Tor_i^R(M,M)=0$ for $i>0$, and

\item[$(3)$]  $M \otimes_R M$ has no embedded primes.
\end{enumerate}
Then $M \cong R$.
\end{thm}

\begin{proof}
We proceed by contradiction. Suppose that $\bigwedge^2_R(M) \ne 0$ so,  by Fact \ref{split}, we have $M \otimes_R M \cong \bigwedge^2_R(M) \oplus S^2_R(M)$. Since $M \otimes_R M$ has no embedded primes, it follows that $\bigwedge^2_R(M)$ is has dimension $\dim(R)$. By Lemma \ref{hilbertformula}, we have 
\[H_R(t)H_R(t^2)H_{\bigwedge^2_R(M)}(t)=H_M(t)^2H_R(t^2)-H_M(t^2)H_R(t)^2.\] As each module in question has maximal dimension, we may clear denominators to obtain a formula for multiplicity polynomials with respect to any system of parameters $\underline{x}$ of $R$ 
\[\varepsilon^{\underline{x}}_R(t)\varepsilon^{\underline{x}}_R(t^2)\varepsilon^{\underline{x}}_{\bigwedge^2_R(M)}(t)=\varepsilon^{\underline{x}}_M(t)^2\varepsilon^{\underline{x}}_R(t^2)-\varepsilon^{\underline{x}}_M(t^2)\varepsilon^{\underline{x}}_R(t)^2.\]
Evaluating these at $t=1$ shows that $\varepsilon_{\bigwedge^2_R(M)}(1)=0$, a contradiction.  Therefore, $\bigwedge^2_R(M)=0$, and it follows that $M$ is cyclic. Thus $M \cong R/I$ for some homogeneous ideal $I$.  As $I/I^2 \cong \Tor^R_1(R/I,R/I) \cong \Tor^R_1(M,M)=0$, it follows that $I=0$, concluding the~proof. \end{proof}

\begin{remark}

Hypothesis $(1)$ in Theorem \ref{bigthm} follows from the condition that $e_R(M)=e(R)$ in a number of cases, e.g. if $\m^n$ admits a homogeneous minimal reduction for some $n>0$ (which occurs, for instance, if $R$ is standard graded), if $R$ is Artinian, if $M$ has constant rank, or if $R$ is Cohen-Macaulay and $M=\w_R$ \cite[Corollary 4.4.6]{BH93}.
When $M$ has rank, hypothesis $(1)$ in Theorem \ref{bigthm} is equivalent to $M$ having rank $1$. Frequently, such modules are ideals, e.g., if $R$ is a domain and $M$ is torsion-free. 

\end{remark}

In what follows, we set $(-)^{\vee}=\Hom_R(-,\w_R)$

\begin{cor}\label{extcor}

Let $R$ be a positively graded Cohen-Macaulay $k$-algebra with $\chara k \ne 2$. If $M$ is a finitely generated graded maximal Cohen-Macaulay $R$-module such that:

\begin{enumerate}

\item[$(1)$]   For some homogeneous system of parameters $\underline{x}$ of $R$, $\varepsilon^{\underline{x}}_M(1)=\varepsilon^{\underline{x}}_R(1)$.

\item[$(2)$]  $\Ext^i_R(M,M^{\vee})=0$ for $i>0$.
\end{enumerate}
Then $M \cong R$.

\end{cor}

\begin{proof}

From \cite[Lemma 3.4 (1)]{LM20}, we have that $M \otimes_R M$ is maximal Cohen-Macaulay and that $\Tor_i^R(M,M)=0$ for all $i>0$. The result then follows from Theorem \ref{bigthm}. \end{proof}

As an immediate consequence of Corollary \ref{extcor} (cf. \cite[Corollary 4.4.6]{BH93}), we prove the (commutative) graded case of the Tachikawa conjecture in charactersitic different from $2$, which also follows from work of Zeng. Notably, Zeng's approach requires passing to noncommutative algebras, whereas our proof uses only techniques in commutative algebra.   

\begin{cor}[{\cite[Theorem 1.3]{Zeng90}}]\label{cor200627a}
Let $R$ be a positively graded Cohen-Macaulay $k$-algebra with $\chara k \ne 2$. If $\Ext^i_R(\omega_R,R)=0$ for every $i>0$, then $R$ is Gorenstein. 
\end{cor}

For Artinian rings, condition~(1) of Theorem~\ref{bigthm} follows from the hypothesis that $l_R(M)=l(R)$ and condtion~(3) of Theorem~\ref{bigthm} is automatically satisfied, so we obtain the following.

\begin{cor}\label{selfTorInd}\label{main}
Let $R$ be an Artinian positively graded $k$-algebra  with $\chara k \ne 2$. If $M$ is a finitely generated graded $R$-module such that  $l_R(M)=l(R)$ and $\Tor_i^R(M,M)=0$ for every $i>0$, then $M\cong R$.
\end{cor}


\section*{Acknowledgments}

We thank the reviewer for their suggestions for the improvement of this work. We would like to thank Luchezar Avramov and Srikanth Iyengar for helpful discussions on Tor-persistence. We thank Ryo Takahashi for helpful feedback, in particular for pointing out  \cite{Zeng90} to us. This work began during the Thematic Program in Commutative Algebra and its Interaction with Algebraic Geometry at the University of Notre Dame. We thank the organizers for a lovely conference.  The second author is  supported by NSF Grant DMS \#2001645.

\bibliographystyle{plain}
\bibliography{mybib}

\end{document}